\documentclass{amsart}\usepackage{amsmath}\usepackage{amscd}\usepackage{amssymb}\usepackage{amsfonts}
\newtheorem{theorem}{Theorem}[section]
\newtheorem{lemma}[theorem]{Lemma}

\newtheorem{proposition}[theorem]{Proposition}

\theoremstyle{definition}

\theoremstyle{remark}

\numberwithin{equation}{section}

\begin{document}
\title[$L^p$-Liouville for SMMS]{On $L^p$-Liouville property for smooth metric measure spaces}
\author{Jia-Yong Wu}
\address{Department of Mathematics, Shanghai Maritime University, Haigang Avenue 1550, Shanghai 201306, P. R. China}\email{jywu81@yahoo.com}
\author{Peng Wu}
\address{Department of Mathematics, Cornell University, Ithaca, NY 14853, USA}\email{wupenguin@math.cornell.edu}
\date{September 15, 2014}
\thanks{}
\begin{abstract}
In this short paper we study $L_f^p$-Liouville property with $0<p<1$ for nonnegative $f$-subharmonic functions on a complete noncompact smooth metric measure space $(M,g,e^{-f}dv)$ with $\mathrm{Ric}_f^m$ bounded below for $0<m\leq\infty$. We prove a sharp $L_f^p$-Liouville theorem when $0<m<\infty$. We also prove an $L_f^p$-Liouville theorem when $\mathrm{Ric}_f\geq 0$ and $|f(x)|\leq \delta(n) \ln r(x)$.
\end{abstract}
\maketitle

\section{Introduction}\label{Int1}

This is a sequel to our previous work \cite{[WuWu1],[WuWu2]}.
For Riemannian manifolds, sharp $L^p$-Liouville theorems with $p>1$, $p=1$, and $0<p<1$ for subharmonic functions were proved by Yau \cite{[Yau]}, Li \cite{[Li]}, Li and Schoen \cite{[Li-Sch]}, respectively. Recently Pigola, Rimoldi, and Setti \cite{[PRS]} proved a sharp $L_f^p$-Liouville theorem with $p>1$ for $f$-subharmonic functions on smooth metric measure spaces.
In \cite{[WuWu1],[WuWu2]}, the authors proved a sharp $L_f^1$-Liouville theorem on smooth metric measure spaces with ($\infty$-)Bakry-\'Emery Ricci curvature $\mathrm{Ric}^{\infty}_f=\mathrm{Ric}+\nabla^2f \geq 0$ by using the $f$-heat kernel estimates on smooth metric measure spaces. For the sake of completeness, in this paper we continue to study $L_f^p$-Liouville theorem with $0<p<1$ on smooth metric measure spaces.

Recall a smooth metric measure space $(M^n,g,e^{-f}dv_g)$ is an $n$-dimensional complete Riemannian manifold $(M,g)$ together with a weighted volume element $e^{-f}dv_g$ for some $f\in C^{\infty}(M)$. The associated $m$-Bakry-\'Emery Ricci curvature \cite{[BE]} is defined as
\[
\mathrm{Ric}_f^m=\mathrm{Ric}+\nabla^2 f-\frac{1}{m}df\otimes df,
\]
for $m\in\mathbb{R}\cup\{\pm\infty\}$. When $0<m<\infty$, the Bochner formula for $m$-Bakry-\'Emery Ricci curvature can be considered as the Bochner formula for the Ricci curvature of an $(n+m)$-dimensional manifold, therefore many analytic and geometric properties for manifolds with Ricci curvature bounded below can be extended to smooth metric measure spaces with $m$-Bakry-\'Emery Ricci curvature bounded below, see for example \cite{[BQ1],[LD],[WW0]} for details. When $m=\infty$, one denotes
\[
\mathrm{Ric}_f=\mathrm{Ric}_f^\infty=\mathrm{Ric}+\nabla^2 f.
\]
In particular, if $\mathrm{Ric}_f=\lambda g$ for some $\lambda\in\mathbb{R}$, then $(M,g,f)$ is a gradient Ricci soliton. 

The $f$-Laplacian $\Delta_f$ on $(M,g,e^{-f}dv)$, which is self-adjoint with respect to the weighted volume element, is defined by
\[
\Delta_f=\Delta-\nabla f\cdot\nabla.
\]
A function $u$ is called $f$-harmonic if $\Delta_f u=0$, and $f$-subharmonic if $\Delta_fu\geq 0$. 
The weighted $L^p$-norm (or $L_f^p$-norm) is defined as
\[
\|u\|_p=\left(\int_M |u|^pe^{-f}dv\right)^{1/p}
\]
for any $0<p<\infty$. We say that $u$ is $L_f^p$-integrable, i.e. $u\in L_f^p$, if $\|u\|_p<\infty$.

\vspace{0.4cm}

Recall the sharp $L^p$-Liouville theorem with $0<p<1$ for Riemannian manifolds proved by Li and Schoen \cite{[Li-Sch]},
\begin{theorem}[Li and Schoen \cite{[Li-Sch]}]\label{Li-Sch}
Let $(M^n,g)$ be an $n$-dimensional complete noncompact Riemannian manfold. There exists a constant $\delta(n)>0$ depending only on $n$ such that, if
\[
\mathrm{Ric}\geq-\delta(n)r^{-2}(x),
\]
as $r(x)\to \infty$, then any nonnegative $L^p$-integrable subharmonic function with $0<p<1$ must be identically zero.
\end{theorem}
Li and Schoen \cite{[Li-Sch]} constructed an explicit example showing that the curvature assumption in Theorem \ref{Li-Sch} is sharp. 
\vspace{0.4cm}

Motivated by Theorem \ref{Li-Sch}, we first prove a sharp $L_f^p$-Liouville theorem with $0<p<1$ for smooth metric measure spaces with $\mathrm{Ric}_f^m$ bounded below,
\begin{theorem}\label{liouv1}
Let $(M^n,g,e^{-f}dv)$ be an $n$-dimensional complete noncompact smooth metric measure space. For any $0<m<\infty$, there exists a constant $\delta(n, m)>0$ depending only on $n$ and $m$ such that, if
\[
\mathrm{Ric}^m_f\geq-\delta(n, m)\,r^{-2}(x),
\]
as $r(x)\to \infty$, then any nonnegative $L_f^p$-integrable $f$-subharmonic function with $0<p<1$ must be identically zero.
\end{theorem}

In Section \ref{sec4}, we show, by constructing an explicit example, that Theorem \ref{liouv1} is sharp, in fact in our example $\mathrm{Ric}^m_f\approx -a\,r^{-2}(x)$ for some $a$ large enough.
\vspace{0.4cm}

For smooth metric measure spaces with $\mathrm{Ric}_f$ bounded below, the first author  \cite{[Wu]} proved an $L_f^p$-Liouville theorem with $0<p<1$ when $f$ is bounded and $\mathrm{Ric}_f\geq-\delta(n)\,r^{-2}(x)$ as $r(x)\to \infty$. In this paper we prove
\begin{theorem}\label{liouv2}
Let $(M^n,g,e^{-f}dv)$ be a complete noncompact smooth metric measure spaces with $\mathrm{Ric}_f\geq 0$. Then there exists a constant $\delta(n)$ depending only on $n$ such that, if
\[
|f(x)|\leq\delta(n)\ln r(x),
\]
as $r(x)\to \infty$, then any nonnegative $L_f^p$-integrable $f$-subharmonic function with $0<p<1$ must be identically zero.
\end{theorem}
We expect Theorem \ref{liouv2} to be sharp too.

\vspace{0.5cm}

This rest of the paper is organized as follows. In Section \ref{sec2}, we prove Theorem \ref{liouv1} following the argument of Li and Schoen. In Section \ref{sec3}, we prove Theorem \ref{liouv2} by combining an $f$-mean value inequality \cite{[WuWu1],[WuWu2]} and an argument of Yau \cite{[Yau2]}. In Section \ref{sec4}, we construct an explicit example to illustrate the sharpness of Theorem \ref{liouv1}.

\textbf{Acknowledgement}.
Part of the work was done when the first author was visiting the Department of Mathematics at Cornell University, he greatly thanks Professor Xiaodong Cao for his help and the department for their hospitality.
The first author was partially supported by NSFC (11101267, 11271132) and the China Scholarship Council (201208310431). The second author was partially supported by an AMS-Simons postdoctoral travel grant.


\section{Proof of Theorem \ref{liouv1}}\label{sec2}

\begin{proof}[Proof of Theorem \ref{liouv1}]

Following the weighted Laplacian comparison theorem (for example \cite{[BQ2]}), the weighted volume comparison theorem, and an argument similar to that in the proof of Theorem 2.1 in Li and Schoen \cite{[Li-Sch]} (see also Theorem 5.3 in \cite{[Wu]}), we get a mean value inequality for $f$-subharmonic functions, if
$\mathrm{Ric}^m_f\geq-(n+m-1)K(x,5R)$, then
\begin{equation}\label{kbds}
\sup_{B_x(R/2)}u^p\leq  \frac{e^{c(1+\sqrt{K(x,5R)}R)}}{V_f(B_x(R))}
\int_{B_x(R)}u^p e^{-f}dv
\end{equation}
for nonnegative $f$-subharmonic function $u$ on $B_x(5R)$, where $c$ depends only on $n$, $m$ and $p$.

We will show that $u$ must vanish at infinity if $u\in L_f^p$ by proving that the growth rate of the $f$-volume is large enough under the assumption using the weighted volume comparison theorem, and by the maximum principle $u$ must be identically zero.

Let $x\in M$ and $\gamma$ be a minimal geodesic from $o$ to $x$, $\gamma(0)=o$ and $\gamma(T)=x$. Define $t_i\in[O,T]$, $0\leq i\leq k$, such that
\[
t_0=0,\quad t_1=1+\beta,\quad
\ldots,\quad t_i=2\sum^i_{j=0}\beta^j-1-\beta^i,
\]
for some $\beta>1$ to be determined later, and $k$ to be the number such that $t_k<T$ and $t_{k+1}\geq T$. Denote $x_i=\gamma(t_i)$, so they satisfy
\[
r(x_i,x_{i+1})=\beta^i+\beta^{i+1},\quad r(o,x_i)=t_i\quad
\mathrm{and}\quad r(x_k,x)<\beta^k+\beta^{k+1}.
\]
Moreover, the geodesic balls $B_{x_i}(\beta^i)$ cover $\gamma([0,2\sum^k_{j=0}\beta^j-1])$ and their interiors are disjoint.

We claim
\begin{equation}\label{kbdenshi}
V_f(B_{x_k}(\beta^k))\geq C
\left(\frac{\beta^{n+m}}{(\beta+2)^{n+m}-\beta^{n+m}}\right)^k
V_f(B_o(1))
\end{equation}
for a fixed
\[
\beta>\frac{2}{2^{\frac{1}{n+m}}-1}>1.
\]

Proof of the claim. For each $1\leq i\leq k$, by the relative comparison theorem (see (4.10) in \cite{[WW]}), we have
\begin{equation*}
\begin{aligned}
V_f(B_{x_i}(\beta^i))&\geq D_i\left[V_f(B_{x_i}(\beta^i+2\beta^{i-1}))
-V_f(B_{x_i}(\beta^i))\right]\\
&\geq D_iV_f(B_{x_{i-1}}(\beta^{i-1})),
\end{aligned}
\end{equation*}
where
\[
D_i=\frac{\int^{\beta^i\sqrt{K(x_i,\beta^i+2\beta^{i-1})}}_0\sinh^{n+m-1} tdt}
{\int^{(\beta^i+2\beta^{i-1})\sqrt{K(x_i,\beta^i+2\beta^{i-1})}}_{
\beta^i\sqrt{K(x_i,\beta^i+2\beta^{i-1})}}\sinh ^{n+m-1}tdt},
\]
since
\[
\mathrm{Ric}^m_f\geq-(n+m-1)K(x_i,\beta^i+2\beta^{i-1})
\]
on $B_{x_i}(\beta^i+2\beta^{i-1})$. Therefore,
\begin{equation}\label{sum}
V_f(B_{x_k}(\beta^k))\geq V_f(B_o(1))\prod^k_{i=1}D_i.
\end{equation}
Since $r(o,x_i)=2\sum^i_{j=0}\beta^j-1-\beta^i$, the curvature assumption implies that
\begin{equation*}
\begin{aligned}
\sqrt{K(x_i,\beta^i+2\beta^{i-1})}
&\leq\sqrt{\delta(n,m)}\left(2\sum^{i-2}_{j=0}\beta^j-1\right)^{-1}\\
&=\sqrt{\delta(n,m)}\,\frac{\beta-1}{2\beta^{i-1}-\beta-1}
\end{aligned}
\end{equation*}
for sufficiently large $i$. Hence
\begin{equation*}
\begin{aligned}
\beta^i\sqrt{K(x_i,\beta^i+2\beta^{i-1})}
\leq&\sqrt{\delta(n,m)}\frac{(\beta-1)\beta}{2-\beta^{2-i}-\beta^{1-i}}
\end{aligned}
\end{equation*}
which can be made arbitrarily small for a fixed $\beta>2\,(2^{1/(n+m)}-1)^{-1}>1$ by choosing $\delta(n,m)$ to be sufficiently small. So we get
\begin{equation*}
\begin{aligned}
D_i&\approx\frac{(\beta^i)^{n+m}}{(\beta^i+2\beta^{i-1})^{n+m}-(\beta^i)^{n+m}}\\
&=\frac{\beta^{n+m}}{(\beta+2)^{n+m}-\beta^{n+m}}
\end{aligned}
\end{equation*}
by simply approximating $\sinh t$ by $t$. Plugging into \eqref{sum} we proved the claim.

\vspace{0.5em}

Next we estimate $V_f(B_x(\beta^{k+1}))$, which will be divided into two cases.

Case 1: $r(x,x_k)\leq \beta^k(\beta-1)$. So $B_{x_k}(\beta^k)\subset B_x(\beta^{k+1})$, and
\[
V_f(B_{x_k}(\beta^k))\leq V_f(B_x(\beta^{k+1})).
\]
By \eqref{kbdenshi}, we get
\[
V_f(B_x(\beta^{k+1}))\geq C\left(\frac{\beta^{n+m}}
{(\beta+2)^{n+m}-\beta^{n+m}}\right)^kV_f(B_o(1)).
\]

Case 2: $r(x,x_k)>\beta^k(\beta-1)$. So $B_{x_k}(\beta^k)\subset B_x\big(r(x,x_k)+\beta^k\big)\backslash B_x\big(r(x,x_k)-\beta^k\big)$.
By the relative comparison theorem, we have
\begin{equation*}
\begin{aligned}
V_f(B_x(\beta^k))&\geq D\left[V_f(B_x\big(r(x,x_k)+\beta^k\big)
-V_f(B_x\big(r(x,x_k)-\beta^k\big)\right]\\
&\geq DV_f(B_{x_k}(\beta^k)),
\end{aligned}
\end{equation*}
where
\[
D=\frac{\int^{\beta^k\sqrt{K(x,r(x,x_k)+\beta^k)}}_0\sinh^{n+m-1} tdt}
{\int^{(r(x,x_k)+\beta^k)\sqrt{K(x,r(x,x_k)+\beta^k)}}_{
(r(x,x_k)-\beta^k)\sqrt{K(x,r(x,x_k)+\beta^k)}}\sinh^{n+m-1} tdt}.
\]
Since
\begin{equation*}
\begin{aligned}
(r(x,x_k)+\beta^k)\sqrt{K(x,r(x,x_k)+\beta^k)}&\leq
(\beta^{k+1}+2\beta^k)\sqrt{K(x,r(x,x_k)+\beta^k)}\\
&\leq\frac{\sqrt{\delta(n,m)}}{2}\cdot\beta(\beta-1)
\end{aligned}
\end{equation*}
can be made sufficiently small, so we get
\[
D\approx\frac{\beta^{n+m}}{(\beta+2)^{n+m}}.
\]
Combining with \eqref{kbdenshi} we get
\begin{equation*}
\begin{aligned}
V_f(B_x(\beta^{k+1}))&\geq\frac{C\beta^{n+m}}{(\beta+2)^{n+m}}
\left(\frac{\beta^{n+m}}
{(\beta+2)^{n+m}-\beta^{n+m}}\right)^{k+1}V_f(B_o(1))\\
&\geq \tilde{C}\left(\frac{\beta^{n+m}}
{(\beta+2)^{n+m}-\beta^{n+m}}\right)^kV_f(B_o(1)),
\end{aligned}
\end{equation*}
where $\tilde{C}$ depends only on $n$ and $\beta$.

Therefore, we have
\begin{equation}
\begin{aligned}\label{keyineq}
V_f(B_x(\beta^{k+1}))\geq C\left(\frac{\beta^{n+m}}
{(\beta+2)^{n+m}-\beta^{n+m}}\right)^kV_f(B_o(1)).
\end{aligned}
\end{equation}

Next let $r(x)\to \infty$, then $k\to \infty$. By the choice of $\beta$,
\[
\frac{\beta^{n+m}}{(\beta+2)^{n+m}-\beta^{n+m}}>1,
\]
so the right hand side of \eqref{keyineq} approaches to infinity. Let $R=\beta^{k+1}$,  by the assumption, $R\sqrt{K(x,5R)}$ is bounded from above. Therefore by \eqref{kbds}, we have
\begin{equation}\label{contr}
u^p(x)\leq CV^{-1}_f(B_x(R)),
\end{equation}
where $C$ depends on $n$, $m$, $p$ and $\|u\|_p$. Therefore $u(x)\to 0$ as $r(x)\to\infty$, and Theorem \ref{liouv1} follows from the maximum principle.
\end{proof}


\section{Proof of Theorem \ref{liouv2}}\label{sec3}

First recall the $f$-volume comparison theorem proved by Wei and Wylie \cite{[WW]},
\begin{theorem}\label{comp}
Let $(M^n,g,e^{-f}dv)$ be an $n$-dimensional complete noncompact smooth metric measure space. If $\mathrm{Ric}_f\geq0$, then for any $x\in B_o(R)$,
\[
\frac{V_f(B_x(R_1,R_2))}{V_f(B_x(r_1,r_2))}\leq
e^{4A}\frac{R^n_{2}-R^n_{1}}{r^n_{2}-r^n_{1}}
\]
for any $0<r_1<r_2,\ 0<R_1<R_2<R$, $r_1\leq R_1,\ r_2\leq R_2<R$, where $B_x(R_1,R_2)=B_x(R_2)\backslash B_x(R_1)$, and $A(R)=\sup_{x\in B_o(3R)}|f|(x)$.
\end{theorem}

Applying Theorem \ref{comp}, we get the following  $f$-volume growth estimate.
\begin{proposition}\label{volest}
Let $(M,g,e^{-f}dv)$ be an $n$-dimensional complete noncompact smooth metric measure space with $\mathrm{Ric}_f\geq 0$. If there exists a sufficiently small constant $\delta(n)>0$ depending only on $n$ such that
\[
|f(x)|\leq \delta(n) \ln r(x),
\]
as $r(x)\to \infty$. Then there exists a constant $C>0$ such that for $r$ sufficiently large,
\[
V_f(B_o(r))\geq C\,r^{1-\delta(n)}.
\]
\end{proposition}
\begin{proof}
Let $x\in M$ be a point with $d(x,o)=r\geq2$. Choosing $R_2=r+1$, $R_1=r-1$, $r_2=r-1$ and $r_1=0$, by Theorem \ref{comp} we have
\begin{equation*}
\begin{aligned}
\frac{V_f(B_x(r+1))-V_f(B_x(r-1))}{V_f(B_x(r-1))}&\leq c(n)(r+1)^{\delta(n)}\,\frac{(r+1)^n-(r-1)^n}{(r-1)^n}\\
&\leq\frac{C(n)}{r^{1-\delta(n)}}
\end{aligned}
\end{equation*}
Since $B_o(1)\subset B_x(r+1)\setminus B_x(r-1)$ and $B_x(r-1)\subset B_o(2r-1)$, therefore we have
\begin{equation*}
\begin{aligned}
V_f(B_o(2r-1))&\geq V_f(B_x(r-1))\\
&\geq\frac{V_f(B_o(1))}{C(n)}r^{1-\delta(n)}.
\end{aligned}
\end{equation*}
\end{proof}

Next recall a weighted mean value inequality on complete smooth metric measure spaces proved by the authors \cite{[WuWu2]}.
\begin{lemma}\label{mevain}
Let $(M,g,e^{-f}dv)$ be an $n$-dimensional complete noncompact smooth metric measure space with $\mathrm{Ric}_f\geq 0$. Let $u$ be a smooth positive subsolution to the $f$-heat equation in $B_o(R)$. For $0<p<\infty$, there exist constants $c_1(n,p)$ and $c_2(n,p)$ such that,
\[
\sup_{B_o(R/2)}u^p\leq
\frac{c_1e^{c_2A}}{V_f(B_o(R))}\int_{B_o(R)}u^p\,\,\, e^{-f} dv.
\]
\end{lemma}

\vspace{1em}

Now we apply Lemma \ref{mevain} and Proposition \ref{volest} to prove Theorem \ref{liouv2}.
\begin{proof}[Proof of Theorem \ref{liouv2}]
Under the assumption, by Lemma \ref{mevain} we have the following mean value inequality
\[
\sup_{B_o(R/2)}u^p\leq  \frac{c_1\,R^{c_2\delta(n)}}{V_f(B_o(R))}
\int_{B_o(R)}u^p e^{-f}dv
\]
where constant $c_1$ and $c_2$ depend on $n$ and $p$. By Proposition \ref{volest}, we have
\begin{equation}\label{contr2}
\sup_{B_o(R/2)}u^p\leq \frac{C}{R^{1-c_3\delta(n)}},
\end{equation}
where constant $c_3$ depends on $n$ and $p$, and $C$ depends on $n$, $p$ and the $L_f^p$-norm of $u$. Taking $R\rightarrow\infty$, then $u(x)\to 0$ as long as $\delta(n)<c_3^{-1}$, and Theorem \ref{liouv2} follows.
\end{proof}


\section{An example}\label{sec4}

In this section, we construct an example to illustrate the sharpness of Theorem \ref{liouv1}. 

Consider the Euclidean space $(\mathbb{R}^2,g_0)$. Let the potential function be $f=a \ln r$ when $r\geq 2$ for some constant $a>0$, and extend it smoothly to $r<2$. The $f$-Laplacian is given by
\[
\Delta_fu=u_{rr}+\Big(\frac 1r-f'\Big)u_r+\frac{1}{r^2}u_{\theta\theta}.
\]

It is a direct computation that, when $r\geq 2$
\begin{equation*}
\begin{split}
{\mathrm{Ric}^m_f}_{11}=&f''-\frac 1m f'^2=-a\left(1+\frac am\right)\frac {1}{r^2},\\
{\mathrm{Ric}^m_f}_{22}=&\frac 1r f'=\frac{a}{r^2},\\
{\mathrm{Ric}^m_f}_{12}=&0,
\end{split}
\end{equation*}
that is
\[
\mathrm{Ric}^m_f\geq-a\left(1+\frac am\right)r^{-2}.
\]

It is easy to verify that $u(r)=r^a$ is an $f$-harmonic function with
\begin{equation}
\begin{aligned}\label{lpnorm}
\|u\|_p^p&=\int^{\infty}_2\int^{2\pi}_0|u|^p\, r\,e^{-f}d\theta dr+\int^2_0\int^{2\pi}_0|u|^p\, r\,e^{-f}d\theta dr\\
&=\int^{\infty}_2\int^{2\pi}_0r^{ap}\, r\,e^{-a\ln r}d\theta dr+C\\
&=2\pi\int^{\infty}_2r^{ap-a+1}dr+C.
\end{aligned}
\end{equation}
For any $0<p<1$, if $a>2/(1-p)$, then the integral is finite and $u$ is an $L^p_f$-integrable $f$-harmonic function on $(\mathbb{R}^2,g_0,r^{-a}dv)$.

\bibliographystyle{amsplain}

\end{document}